\documentclass[12pt]{article}
\usepackage{a4}
\usepackage{amsmath}
\usepackage{amsthm}
\usepackage{amsfonts}
\usepackage{amssymb}
\usepackage{stmaryrd}
\usepackage{cite}
\usepackage{epsfig}
\newtheorem{theorem}{Theorem}
\newtheorem{lemma}[theorem]{Lemma}
\newtheorem{proposition}[theorem]{Proposition}
\newtheorem{corollary}[theorem]{Corollary}
\newtheorem{claim}{Claim}
\newtheorem{conjecture}{Conjecture}
\def\Aut{{\rm Aut}}
\def\CC{{\cal C}}
\def\JJ{{\cal J}}

\def\NN{{\mathbb N}}
\def\RR{{\mathbb R}}
\def\dif{{\rm d}}

\def\WW {{\cal W}}
\def\url#1{{\tt #1}}
\newcommand{\bit}[1]{\left\langle #1\right\rangle}
\newcommand{\unlab}[2]{\left\llbracket #1\right\rrbracket_{#2}}
\newcommand{\rt}{\right}
\newcommand{\lt}{\left}
\begin{document}
\title{Compactness and finite forcibility of graphons\thanks{The~work leading to this invention has received funding from the European Research Council under the European Union's Seventh Framework Programme (FP7/2007-2013)/ERC grant agreement no.~259385.}}
\author{Roman Glebov\thanks{Department of Computer Science, Ben-Gurion University of the Negev, P.O.B. 653, Beer-Sheva 8410501, Israel.
E-mail: {\tt roman.l.glebov@gmail.com}. Previous affiliations: School of Computer Science and engineering, The Hebrew University of Jerusalem, Jerusalem 91904, Israel, and Mathematics Institute and DIMAP, University of Warwick, Coventry CV4 7AL, UK. This author was also supported by the ERC grant ``High-dimensional combinatorics'' at the Hebrew University.}\and
        Daniel Kr\'al'\thanks{Faculty of Informatics, Masaryk University, Botanick\'a 68A, 602 00 Brno, Czech Republic, and Mathematics Institute, DIMAP and Department of Computer Science, University of Warwick, Coventry CV4 7AL, UK. E-mail: {\tt dkral@fi.muni.cz}. This author was also supported by the European Research Council (ERC) under the European Union’s Horizon 2020 research and innovation programme (grant agreement No 648509) and by the Engineering and Physical Sciences Research Council Standard Grant number EP/M025365/1. This publication reflects only its authors' view; the European Research Council Executive Agency is not responsible for any use that may be made of the information it contains.}\and
	Jan Volec\thanks{Department of Mathematics, Emory University, 30307 Atlanta, USA. E-mail: {\tt jan@ucw.cz}. This author was also supported
from the European Union’s Horizon 2020 research and innovation programme under the Marie Sk\l{}odowska-Curie grant agreement No. 800607.
E-mail: {\tt jan@ucw.cz}.
Previous affiliations: Department of Mathematics, ETH, 8092 Zurich, Switzerland, and Mathematics Institute and DIMAP, University of Warwick, Coventry CV4 7AL, UK.}
	}

\date{}
\maketitle

\begin{abstract}
Graphons are analytic objects associated with convergent sequences of graphs.
Problems from extremal combinatorics and theoretical computer science
led to a study of graphons determined by finitely many subgraph densities,
which are referred to as finitely forcible.
Following the intuition that such graphons should have finitary structure,
Lov\'asz and Szegedy conjectured that the topological space of typical vertices of a finitely forcible graphon is always compact.
We disprove the conjecture by constructing a finitely forcible graphon such that the associated space is not compact.
The construction method gives a general framework for constructing finitely forcible graphons with non-trivial properties.
%In our construction, the space fails to be even locally compact.
\end{abstract}

\section{Introduction}
\label{sec:intro}

Recently, a theory of limits of combinatorial structures emerged and attracted
substantial attention. 
In this paper, we study the case of limits of dense (finite) graphs,
which has been developed in a series of papers primarily by Borgs, Chayes, Lov\'asz,
S\'os, Szegedy and Vesztergombi~\cite{bib-borgs08+,bib-borgs+,bib-borgs06+,bib-lovasz06+,bib-lovasz10+}.
A sequence of graphs is convergent if the density of every graph in the graphs contained in the sequence converges.
A convergent sequence of graphs can be associated with an analytic object
which is a symmetric measurable function from the unit square $[0,1]^2$ to $[0,1]$;
this object is referred to as a graphon.
Graph limits and graphons are also closely related to flag algebras introduced
by Razborov~\cite{bib-razborov07}, which were successfully applied to numerous
problems in extremal combinatorics~\cite{bib-flag1, bib-flag2, bib-flagrecent, bib-balogh14,bib-balogh16,bib-balogh17,bib-balogh15,bib-coregliano17,bib-cumm13,bib-das13, bib-evan15,bib-falgas15,bib-gethner17, bib-flag3half, bib-goaoc15, bib-flag4, bib-flag5, bib-flag6, bib-hladky17, bib-tilings, bib-flag7, bib-flag8, bib-lidicky17+,bib-nikiforov, bib-flag10, bib-flag9, bib-razborov07, bib-flag11, bib-flag12, bib-reiher16}.
The development of the graph limit theory is also reflected in a recent monograph by Lov\'asz~\cite{bib-lovasz-book}.

In this paper, we are concerned with finitely forcible graphons,
i.e., those that are uniquely determined (up to a natural equivalence) by finitely many subgraph densities.
Such graphons are related to uniqueness of optimal configurations
in extremal graph theory as well as to other problems.
For example,
a classical result on the pseudorandomness of graphs~\cite{bib-chung89+,bib-thomason} asserts that
a large graph is pseudorandom if and only if the homomorphic
densities of $K_2$ and $C_4$ are the same as in the Erd\H{o}s-R\'enyi random graph $G_{n,1/2}$.
This result can be cast in the language of graphons as follows:
the graphon identically equal to $1/2$ is uniquely determined
by homomorphic densities of $K_2$ and $C_4$, in particular, it is finitely forcible.
Another example that can be cast in the language of finite forcibility
is the asymptotic version of the theorem of Tur\'an~\cite{bib-turan}:
there exists a unique graphon with edge density $\frac{r-1}{r}$ and zero density of $K_{r+1}$.

The result on the pseudorandomness of graphs mentioned above
was generalized by Lov\'asz and S\'os~\cite{bib-lovasz08+},
who proved that any graphon that is a step function is finitely forcible.
Lov\'asz and Szegedy~\cite{bib-lovasz11+} found further examples of finitely forcible graphons and
studied properties of finitely forcible graphons in general.
In particular,
they observed that the structure of every example of a finitely forcible graphon they had found is somewhat simple.
Formalizing this intuition,
they associated typical vertices of a graphon $W$ with the topological space $T(W)\subseteq L_1[0,1]$ (see Section~\ref{sect-notation} for definitions)
and observed that all their examples of finitely forcible graphons $W$ have compact and at most $1$-dimensional $T(W)$.
This motivated the following conjectures~\cite[Conjectures 9 and 10]{bib-lovasz11+}.

\begin{conjecture}[Lov\'asz and Szegedy]
\label{conj:1}
If $W$ is a finitely forcible graphon, then $T(W)$ is a compact space.
\end{conjecture}

\begin{conjecture}[Lov\'asz and Szegedy]
\label{conj:2}
If $W$ is a finitely forcible graphon, then $T(W)$ is finite dimensional.
\end{conjecture}

In relation to the former of the conjectures,
they noted that they could not even prove that $T(W)$ had to be locally compact.
In this paper, we present a construction of a finitely forcible graphon $W$ such that
$T(W)$ fails to be locally compact, in particular, $T(W)$ is not compact.

\begin{theorem}
\label{thm:main}
There exists a finitely forcible graphon $W_R$ such that the topological space $T(W_R)$ is not locally compact.
\end{theorem}

In order to prove Theorem~\ref{thm:main}, we present a framework for constructing finitely forcible graphons
which uses density constraints and decorated density constraints.
The framework is inspired by the notion of flag algebras and
builds on the arguments used in existing constructions of finitely forcible graphons
such as those in~\cite{bib-lovasz08+,bib-norine-comm} (details related to the construction from~\cite{bib-norine-comm} can be found in Section~\ref{sec:concl}).
This framework can be used to construct finitely forcible graphons with other non-trivial properties.
In particular, it was used in~\cite{bib-our-next-paper} to completely disprove Conjecture~\ref{conj:2} (see Section~\ref{sec:concl} for further details), and
in~\cite{bib-third-paper} to construct finitely forcible graphons with no small weak regularity partitions.
This line of research culminated with~\cite{bib-CKM} and~\cite{bib-KLNS},
where the techniques set out in this paper were used to show that any graphon can be a subgraphon of a finitely forcible graphon.

\section{Notation}
\label{sect-notation}

In this section, we introduce notation related to the concepts used in this paper.
We use $[n]$ to denote the set of the first $n$ positive integers.
A {\em graph\/} is a pair $(V,E)$ where $E\subseteq {V\choose 2}$.
The elements of $V$ are called {\em vertices\/} and the elements of $E$ are called {\em edges\/}.
All graphs considered in this paper are {\em finite\/}, i.e., they have finitely many vertices.
The {\em order\/} of a graph $G$ is the number of its vertices and is denoted by $|G|$.
The {\em density\/} $d(H,G)$ of a graph $H$ in a graph $G$ is the probability that
a uniformly chosen subset of $|H|$ vertices of $G$ induces a subgraph isomorphic to $H$.
If $|H|>|G|$, we set $d(H,G)=0$.
A sequence of graphs $(G_i)_{i\in\NN}$ is {\em convergent}
if the sequence $(d(H,G_i))_{i\in\NN}$ converges for every graph $H$.

We now present basic notions from the theory of dense graph limits as
developed in~\cite{bib-borgs08+,bib-borgs+,bib-borgs06+,bib-lovasz06+}.
A {\em graphon\/} $W$ is a symmetric measurable function from $[0,1]^2$ to $[0,1]$.
Here, symmetric stands for the property that $W(x,y)=W(y,x)$ for every $x,y\in [0,1]$.
We remark that, throughout the paper, we work with the Lebesgue measure on $\RR^d$ only.
A {\em $W$-random graph\/} of order $k$ is obtained by sampling $k$ random points $x_1,\ldots,x_k\in [0,1]$
uniformly and independently and
joining the $i$-th and the $j$-th vertex by an edge with probability $W(x_i,x_j)$.
Since the points of $[0,1]$ play the role of vertices, we refer to them as to vertices of $W$.
To simplify our notation further, if $A\subseteq [0,1]$ is measurable, we use $|A|$ for its measure.
The {\em density $d(H,W)$\/} of a graph $H$ in a graphon $W$ is equal to the probability that
a $W$-random graph of order $|H|$ is isomorphic to $H$.
In particular, the following holds:
\[d(H,W)=\frac{|H|!}{|\Aut(H)|}\int\limits_{[0,1]^{|H|}} \prod_{(i,j)\in E(H)} W(x_i,x_j) \prod_{(i,j)\not\in E(H)} (1-W(x_i,x_j)) \,\dif x_1\cdots\dif x_{|H|},\]
where $\Aut(H)$ is the automorphism group of $H$.
One of the key results in the theory of dense graph limits~\cite{bib-lovasz06+} asserts that
for every convergent sequence $(G_i)_{i\in\NN}$ of graphs with increasing orders,
there exists a graphon $W$ (called the {\em limit} of the sequence) such that 
\[d(H,W)=\lim_{i\to\infty} d(H,G_i)\]
for every graph $H$.
On the other hand, every graphon $W$ is a limit of a convergent sequence of graphs
since the sequence of $W$-random graphs with increasing orders converges with probability one and its limit is $W$.

Every graphon can be assigned a topological space corresponding to its typical vertices~\cite{bib-lovasz10++}.
For a graphon $W$,
define for $u\in [0,1]$ a function $f^W_u(y)=W(u,y)$.
For an open set $A\subseteq L_1[0,1]$, we write $A^W$ for $\lt\{u\in [0,1],\, f_u^W\in A\rt\}$.
Let $T(W)$ be the set formed by the functions $f\in L_1[0,1]$ such that $\lt|U^W\rt|>0$ for every neighborhood $U$ of $f$.
The set $T(W)$ inherits topology from $L_1[0,1]$.
The vertices $u\in [0,1]$ with $f^W_u\in T(W)$ are called {\em typical vertices\/} of a graphon $W$.
Notice that almost every vertex is typical~\cite{bib-lovasz10++}.

Two graphons $W_1$ and $W_2$ are {\em weakly isomorphic} if $d(H,W_1)=d(H,W_2)$ for every graph $H$.
If $\varphi:[0,1]\to[0,1]$ is a measure preserving map,
then the graphon $W^\varphi(x,y):=W(\varphi(x),\varphi(y))$ is weakly isomorphic to $W$.
The opposite is true in the following sense~\cite{bib-borgs10+}:
if two graphons $W_1$ and $W_2$ are weakly isomorphic,
then there exist measure measure preserving maps $\varphi_1:[0,1]\to [0,1]$ and $\varphi_2:[0,1]\to [0,1]$
such that $W_1^{\varphi_1}(x,y)=W_2^{\varphi_2}(x,y)$ for almost every $(x,y)\in [0,1]^2$.

A graphon $W$ is {\em finitely forcible\/} if there exist graphs $H_1,\ldots,H_k$ such that
every graphon $W'$ satisfying $d(H_i,W)=d(H_i,W')$ for all $i\in\{1,\ldots,k\}$ is weakly isomorphic to $W$.
For example, the result of Diaconis, Homes, and Janson~\cite{DiHoJa8} asserts that
the half graphon $W_{\Delta}(x,y)$ defined as $W_{\Delta}(x,y)=1$ if $x+y\ge 1$, and
$W_{\Delta}=0$, otherwise, is finitely forcible.
Also see~\cite{bib-lovasz11+} for further results.

When establishing that a graphon $W$ is finitely forcible,
it is possible, instead of presenting the list of graphs $H_1,\ldots,H_k$ and their densities as given in the previous paragraph,
to present a set of constraints on densities of graphs in $W$ such that $W$ is the unique graphon satisfying these constraints.
We now formalize this.
A {\em constraint\/} is an equality between two density expressions.
A {\em density expression\/} is a formal polynomial in graphs with real coefficients,
i.e., a real number or a graph $H$ are density expressions, and
if $D_1$ and $D_2$ are two density expression, then the sum $D_1+D_2$ and the product $D_1\cdot D_2$ are also density expressions.
The evaluation of a density expression with respect to a graphon $W$
is the value obtained from the expression by a substition of $d(H,W)$ for every graph $H$ in the expression.
A graphon $W$ {\em satisfies\/} a constraint if both sides of the constraint are evaluated to the same.
If $\CC$ is a finite set of constraints and
$W$ is the unique (up to weak isomorphism) graphon that satisfies all constraints in $\CC$,
then $W$ is finitely forcible.
Indeed, $W$ is the unique (up to weak isomorphism) graphon with densities of graphs appearing in $\CC$
equal to their densities in $W$.
This holds since any graphon with the same densities satisfies all constraints in $\CC$ and
thus it is weakly isomorphic to $W$.

We extend the notion of density expressions to rooted density expressions
based on the ideas from the concept of flag algebras from~\cite{bib-razborov07}.
A graph is {\em rooted\/} if it has $m$ distinguished vertices numbered from $1$ to $m$.
These vertices are referred to as {\em roots\/} while the other vertices are non-roots.
Two rooted graphs are {\em compatible\/} if the subgraphs induced by their roots are isomorphic
through an isomorphism mapping the $i$-th root to the $i$-th root of the other.
Similarly, two rooted graphs are {\em isomorphic\/}
if there exists an isomorphism that maps the $i$-th root of one of them to the $i$-th root of the other.

A {\em rooted density expression\/} is a density expression such that all graphs that appear in the expression
are mutually compatible rooted graphs.
We occasionally speak about compatible rooted density expressions
to emphasize that the rooted graphs in all of them are mutually compatible.
The evaluation of a rooted density expression for a graphon $W$
is a random variable as defined in the next paragraph.

Fix a rooted density expression.
Let $H_0$ be the rooted graph induced by the roots of the graphs appearing in the expression and let $m=|H_0|$.
If $d(H_0,W)=0$, then the density expression evaluates to $0$ with probability one.
If $d(H_0,W)>0$,
we define an auxiliary function $c:[0,1]^m\to [0,1]$ as
\[c(x_1,\ldots, x_m)=  \lt(\prod_{(i,j)\in E(H_0)} W(x_i,x_j)\rt) \cdot\lt(\prod_{(i,j)\not\in E(H_0)} (1-W(x_i,x_j)) \rt).\]
In other words,
the value $c(x_1,\ldots,x_m)$ is the probability that a $W$-random graph is isomorphic to $H_0$ through the identity
conditioned on sampling the points $x_1,\ldots,x_m$ corresponding to the vertices of the $W$-random graph.
We next define a probability measure $\mu$ on $[0,1]^m$.
If $A\subseteq [0,1]^m$ is a measurable set, then:
\[ \mu(A)= \frac{\int\limits_A c(x_1,\ldots, x_m) \dif x_1\cdots\dif x_m}{\int\limits_{[0,1]^m} c(x_1,\ldots, x_m) \dif x_1\cdots\dif x_m}.\]
We now define the {\em evaluation\/} of the rooted density expression.
We sample $m$ points $x_1,\ldots,x_m\in [0,1]$ according to the probability measure $\mu$, and
substitute for every rooted graph $H$ in the expression the probability that
a $W$-random graph of order $|H|$ with the first $m$ vertices numbered form $1$ to $m$ is isomorphic to $H$ 
conditioned on that the points corresponding to these $m$ vertices are $x_1,\ldots,x_m$ and these $m$ vertices induce $H_0$.
When the values of $x_1,\ldots,x_m\in [0,1]$ are fixed,
this probability is equal to
\[\frac{(|H|-m)!}{|\Aut(H)|}\int\limits_{[0,1]^{|H|-m}}
\prod\limits_{(i,j)\in E(H)\setminus\binom{H_0}{2}} W(x_i,x_j)\prod\limits_{(i,j)\not\in E(H)\cup\binom{H_0}{2}} (1-W(x_i,x_j))
\dif x_{m+1}\cdots\dif x_{|H|},\]
where $\Aut(H)$ is the set of automorphisms of $H$ (preserving the roots).
This defines the value of the rooted density expression as a random variable
where the randomness comes from the choice of $x_1,\ldots,x_m$ according to $\mu$.

A {\em rooted constraint\/} is an equality between two rooted density expression $D$ and $D'$ such that
the subgraphs induced by the roots in $D$ and $D'$ are the same.
In particular, $D-D'$ is a rooted density expression.
A graphon $W$ {\em satifies} a rooted constraint $D=D'$
if $D-D'$ evaluates to the random variable equal to $0$ with probability one.

It can be shown (see, e.g., \cite{bib-razborov07}) that
for every rooted density expression $D$ with roots inducing a graph $H_0$,
there exists an ordinary density expression $\unlab{D}{}$ such that the following holds for every graphon $W$:
\begin{itemize}
\item if $d(H_0,W)=0$, then $\unlab{D}{}$ evaluates to $0$ for $W$, and
\item if $d(H_0,W)>0$, then $\unlab{D}{}/d(H_0,W)$ evaluates to the expected value of $D$ for $W$.
\end{itemize}
In particular, the evaluation of $\unlab{D}{}$ for $W$ is equal to the expected value of the evaluation of $D$ for $W$
multiplied by $d(H_0,W)$.
It follows that if $D=D'$ is a rooted constraint,
then $W$ satisfies the rooted constraint $D=D'$ if and only if
$W$ satisfies the (ordinary) constraint $\unlab{(D-D')^2}{}=0$.
Since this allows us to express constraints involving rooted density expressions as ordinary constraints,
we will not distinguish between the two types of constraints in the remainder of the paper.

\section{Partitioned graphons}
\label{sect-partitioned}

In this section, we introduce a notion of partitioned graphons.
Some of the methods presented in this section are
analogous to those used by Lov\'asz and S\'os in~\cite{bib-lovasz08+} and
by Norine~\cite{bib-norine-comm} (see the construction in Section~\ref{sec:concl}).
In particular, they used similar types of arguments to specialize
their constraints to parts of graphons they were forcing as we do in this section.
However, since it is hard to refer to any particular lemma in~\cite{bib-lovasz08+}
instead of presenting a full argument, we give all details.

A {\em degree} of a vertex $u\in [0,1]$ of a graphon $W$ is equal to
\[\int\limits_{[0,1]}W(u,y)\dif y\;\mbox{.}\]
Note that the degree is well-defined for almost every vertex of $W$.
A graphon $W$ is {\em partitioned}
if there exist $k\in\NN$ and positive reals $a_1,\ldots,a_k$ with $\sum_i a_i=1$ and
distinct reals $d_1,\ldots,d_k \in [0,1]$ such that
the the set of vertices of $W$ with degree $d_i$ has measure $a_i$.
We will often speak about partitioned graphons
when having in mind fixed values of $k$, $a_1,\ldots,a_k$, and $d_1,\ldots,d_k$,
which will be clear from the context.
We will also refer to $a_i$ as the size and to $d_i$ as the degree of the $i$-th part.

A specific type of partition can be finitely forced as given in the next lemma.

\begin{lemma}
\label{lm-partition}
Let $k$ be a positive integer,
$a_1,\ldots,a_k$ positive reals summing to one, and
$d_1,\ldots,d_k$ distinct reals between zero and one.
There exists a finite set of constraints $\CC$ such that
a graphon $W$ satisfies $\CC$ if and only if
the set of vertices of $W$ with degree $d_i$ has measure $a_i$.
In other words, if and only if
the graphon $W$ is a partitioned graphon with parts of sizes $a_1,\ldots,a_k$ and degrees $d_1,\ldots,d_k$.
\end{lemma}

\begin{proof}
Consider the following set of constraints:
\[\prod\limits_{i=1}^k (e_1-d_i)=0\;\mbox{, and}\]
\[\unlab{\prod\limits_{i=1,\, i\neq j}^k  (e_1-d_i)}{}=a_j\prod\limits_{i=1,\, i\neq j}^k(d_j-d_i) \mbox{ for every $j\in [k]$,}\]
where $e_1$ is an edge with one root and one non-root.
The first constraint is satisfied if and only if
the degree of almost every vertex is equal to one of the numbers $d_1,\ldots,d_k$.

Next, fix $j\in\{1,\ldots,k\}$ and consider the corresponding contraint on the second line.
The rooted density expression inside the $\unlab{\cdot}{}$-operator evaluates to the random variable that
is equal to 
\[\prod_{i=1,i\not=j}^k(d_j-d_i)\]
if the degree of the root vertex is $d_j$ and
equal to zero if the degree is one of the remaining numbers $d_1,\ldots,d_k$.
Since the left hand side of the constraint evaluates to the expected value of this random variable,
the constraint is satisfied if and only if the vertices of degree $d_j$ have measure $a_j$ (assuming
the degree of almost every vertex is equal to one of the numbers $d_1,\ldots,d_k$).
\end{proof}

Fix a positive integer $k$,
positive reals $a_1,\ldots,a_k$ summing to one, and
distinct reals $d_1,\ldots,d_k\in [0,1]$.
We next consider partitioned graphons with $k$ parts described by these parameters and
write $A_i$ for the $i$-th part, $i=1,\ldots,k$.
When $W$ is a partitioned graphon with such parts,
we use $A_i$ to denote the subset of $[0,1]$ formed by the vertices of degree $d_i$.
A graph $H$ is {\em decorated\/} if each vertex of $H$ is decorated with one of the parts $A_1,\ldots,A_k$;
note that some of the vertices of $H$ may have the same decoration.
The density of a decorated graph $H$ in a graphon $W$ is the probability that
a $W$-random graph with each vertex decorated with the part of $W$ that it belongs to
is isomorphic to $H$ by an isomorphism preserving the decorations.
For example, if $H$ is an edge with its two vertices decorated with $A_1$ and $A_2$,
then the density of $H$ is the density of edges between $A_1$ and $A_2$, i.e.,
\[d(H,W)=\int\limits_{A_1\times A_2}W(x,y)\,\dif x\,\dif y\;\mbox{.}\]

Analogously to the case of non-decorated graphs, we can define {\em rooted decorated graphs\/}.
We require that an isomorphism of two rooted decorated graphs preserves
both the numbering of the roots and the decorations of all the vertices.
In particular, two rooted decorated graphs are compatible
if the subgraphs induced by their roots are isomorphic through such an isomorphism.
In this way, we arrive at the notions of {\em rooted decorated density expressions\/} and {\em decorated constraints\/}.
The evaluation of a rooted decorated density expression is defined in the same way as
the evaluation of a rooted density expression, however,
we additionally condition on the root vertices to be chosen from the parts given by the decorations.

The next lemma shows that 
the expressive power of decorated constraints for partitioned graphons
is the same as that of non-decorated constraints.
We will always use the lemma when it is guaranteed that the considered graphons are partitioned.
Before stating and proving the lemma,
we introduce a convention for drawing density expressions:
edges of graphs are always drawn solid, non-edges dashed, and
if two vertices are not joined, then the picture represents the sum over both possibilities.
If a graph contains some roots, the roots are depicted
by square vertices, and the non-root vertices by circles.
Decorations of vertices are always drawn inside vertices.
If there are more roots from the same part, then the squares are rotated to distinguish the roots.
An example of this convention can be found in Figure~\ref{fig-lm-partitioned}.

\begin{lemma}
\label{lm-partitioned}
Let $k$ be a positive integer, $a_1,\ldots,a_k$ positive reals summing to one and
$d_1,\ldots,d_k$ distinct reals between zero and one.
For every (rooted or non-rooted) decorated constraint $D=D'$, there exists a non-decorated constraint $E=E'$ such that
a partitioned graphon $W$ with $k$ parts described by $a_1,\ldots,a_k$ and $d_1,\ldots,d_k$
satisfies $D=D'$ if and only if it satisfies $E=E'$.
\end{lemma}

\begin{figure}
\begin{center}
\epsfbox{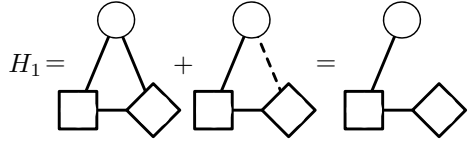}
\end{center}
\caption{The rooted non-decorated graph $H_1$ from the proof of Lemma~\ref{lm-partitioned}
         when $H$ is an edge.}
\label{fig-lm-partitioned}
\end{figure}

\begin{proof}
A rooted decorated constraint can be transformed to an equivalent non-rooted decorated constraint
in the same way that we presented in the non-decorated setting.
Therefore, we assume that the constraint $D=D'$ is non-rooted and
we will construct an equivalent non-decorated constraint $E=E'$.

Let $H$ be a non-rooted decorated graph with vertices $v_1,\ldots,v_n$ such that $v_i$ is labeled with a part $A_{\ell_i}$.
Let $\widetilde{H}$ be the graph $H$ without decorations, and
let $H_i$ be the sum of all rooted non-decorated graphs
on $n+1$ vertices with $n$ roots such that the roots induce $\widetilde{H}$ (with the $j$-th root being $v_j$ for $j=1,\ldots,n$) and
the only non-root is adjacent to $v_i$ in each of the summands.
An example for $i=1$ and $H$ being an edge is given in Figure~\ref{fig-lm-partitioned}.
We claim that the density of $H$ is equal to the following density expression:
\begin{equation}
\label{eq1}
\frac{|H|!}{|\Aut(H)|}\unlab{\prod_{i=1}^n
\prod_{j=1,\, j\neq\ell_i}^k
\frac{H_i-d_j}{d_{\ell_i}-d_j}}{}
\;\mbox{.}
\end{equation}
Indeed, if the density of $\widetilde{H}$ is zero in a graphon $W$,
then the above density expression is zero and so is the density of $H$.
Otherwise, for every choice of the $n$ roots,
the product inside the $\unlab{\cdot}{}$-operator is non-zero if and only if
the $i$-th root belongs to the part $A_{\ell_i}$;
the product is equal to one in such case.
Hence, the value of~\eqref{eq1} is exactly the probability that
uniformly chosen $n$ random vertices induce a labeled copy of $H$ such that the $i$-th vertex belongs to $A_{\ell_i}$.
Let $E$ and $E'$ be non-decorated constraints obtained from $D$ and $D'$, respectively,
by replacing every decorated graph $H$ with the corresponding density expression \eqref{eq1}.
Since $D$ and $E$ evaluate to the same number for every graphon and the same holds for $D'$ and $E'$,
the statement of the lemma follows.
\end{proof}

Since the expressive power of non-decorated and decorated constraints is the same,
we will often drop the adjective decorated in the rest of the paper.

We finish this section with two corollaries of Lemma~\ref{lm-partitioned}.
Informally, the first one says that
it is possible to force a finitely forcible graphon on a part of a partitioned graphon.

\begin{lemma}
\label{lm-gadget}
Let $k$ be a positive integer, $a_1,\ldots,a_k$ positive reals summing to one and
$d_1,\ldots,d_k$ distinct reals between zero and one.
Further, let $W_0$ be a finitely forcible graphon and let $\ell\in [k]$.
There exists a finite set of constraints $\CC$ such that the following holds:
a partitioned graphon $W$ satisfies $\CC$ if and only if
there exist a measure preserving map $\varphi:[0,1]\to [0,1]$ and
a measure preserving map $\varphi'$ from $[0,a_\ell]$ to the $\ell$-part of $W$ such that
\[W_0(\varphi(x),\varphi(y))=W(\varphi'(a_\ell x),\varphi'(a_\ell y))\]
for almost every $x,y\in [0,1]^2$.
Informally, if and only if the subgraphon induced by the $\ell$-th part of $W$ is weakly isomorphic to $W_0$.
\end{lemma}

\begin{proof}
Assume that $W_0$ is forced by $m$ constraints of the form
\[d(H_i,W)=d_i\mbox{ for $i\in [m]$.}\]
The set $\CC$ is formed by the following $m$ constraints
\[d(H'_i,W)=a_{\ell}^{|H_i|}d_i\; \mbox{,}\]
where $H'_i$ is the graph $H_i$ with all vertices decorated with the $A_{\ell}$.
\end{proof}

The second lemma asserts the finite forcibility of a constant edge density between parts of a partitioned graphon.

\begin{lemma}
\label{lm-pseudobipartite}
Let $k$ be a positive integer, $a_1,\ldots,a_k$ positive reals summing to one and
$d_1,\ldots,d_k$ distinct reals between zero and one.
Further, let $\ell,\ell'\in [k]$ and $p\in [0,1]$.
There exists a finite set of constraints $\CC$ such that the following holds:
a partitioned graphon $W$ satisfies $\CC$ if and only if
$W(x,y)=p$ for almost every $x$ and $y$ from the $\ell$-th and $\ell'$-th parts, respectively.
\end{lemma}

\begin{figure}
\begin{center}
\epsfbox{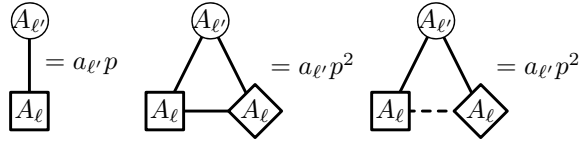}
\end{center}
\caption{The constraints used in the proof of Lemma~\ref{lm-pseudobipartite}.}
\label{fig:CS}
\end{figure}

\begin{proof}
We first define rooted decorated graphs $H$, $H_1$ and $H_2$.
The graph $H$ is an edge with one root;
the root is decorated by $A_{\ell}$ and the other vertex by $A_{\ell'}$.
The graph $H_1$ is a triangle with two roots;
both roots are decorated by $A_{\ell}$ and the remaining vertex by $A_{\ell'}$.
Finally, $H_2$ is a cherry (a path on three vertices) rooted at the two non-adjacent vertices;
the two roots are decorated by $A_{\ell}$ and the remaining vertex by $A_{\ell'}$.

The set $\CC$ is formed by the follwoing three constraints: $H=p$, $H_1=p^2$, and $H_2=p^2$.
The three constraints are depicted in Figure~\ref{fig:CS}.
A graphon $W$ satisfies the three constraints if and only if
\begin{equation}
\int\limits_{A_{\ell'}}W(x,y)\,\dif y=a_{\ell'}p\qquad\mbox{and}\qquad
\int\limits_{A_{\ell'}}W(x,y)\cdot W(x',y)\,\dif y=a_{\ell'}p^2
\label{eq5a}
\end{equation}
for almost every $x$ and $x'$ from $A_\ell$.
In particular, if $W(x,y)=p$ for almost every $(x,y)\in A_{\ell}\times A_{\ell'}$,
then $W$ satisfies the three constraints in $\CC$.

For the ``if'' part of the statement of the lemma,
we follow the reasoning given in~\cite[proof of Lemma 3.3]{bib-lovasz11+} and
conclude that the second equation in \eqref{eq5a} implies that
\begin{equation}
\int\limits_{A_{\ell'}}W^2(x,y)\,\dif y=a_{\ell'}p^2
\label{eq5b}
\end{equation}
for almost every $x$ from the $\ell$-th part of $W$.
The Cauchy-Schwarz inequality used with the first equation from \eqref{eq5a} and the equation from \eqref{eq5b}
yields that, for almost every $x\in A_{\ell}$,
it holds that $W(x,y)=p$ for almost every $y\in A_{\ell'}$.
\end{proof}

\section{Rademacher Graphon}

In this section, we introduce a graphon $W_{R}$, which we refer to as {\em Rademacher graphon},
from the statement of Theorem~\ref{thm:main}.
The name of the graphon comes from the fact that the adjacencies between its parts $A$ and $C$
resembles Rademacher system of functions (such adjacencies also appear in~\cite[Example 13.30]{bib-lovasz-book}).
The graphon $W_{R}$ is visualized in Figure~\ref{fig-chess}.
We establish that $W_{R}$ is finitely forcible in the next two sections.

\begin{figure}
\begin{center}
\epsfbox{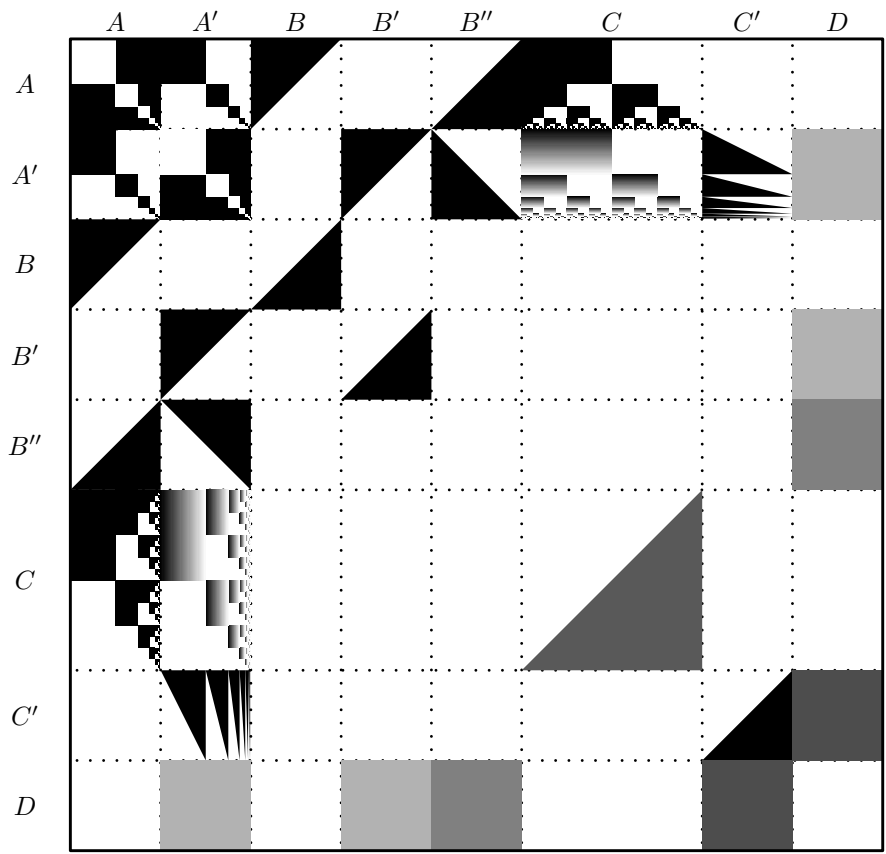}
\end{center}
\caption{Rademacher graphon $W_{R}$.}
\label{fig-chess}
\end{figure}

The graphon $W_{R}$ has eight parts.
Instead of using $A_1, \dots, A_8$ for its parts,
we use $A$, $A'$, $B$, $B'$, $B''$, $C$, $C'$ and $D$.
All the parts except for $C$ have the same size $a=1/9$; the size of $C$ is $2a=2/9$.
Let $\ell_{A}$, $\ell_{A'}$, $\ell_{B}$, $\ell_{B'}$, $\ell_{B''}$, $\ell_{C}$, $\ell_{C'}$ and $\ell_{D}$
be $0$, $a$, $2a$, $3a$, $4a$, $5a$, $7a$ and $8a$, respectively.
The part $Z\in\{A,A',B,B',B'',C',D\}$ of $W_R$ will be formed by the interval $[\ell_Z,\ell_Z+a)$ and
the part $C$ by the interval $[\ell_C,\ell_C+2a)$.

\begin{table}
\begin{center}
\begin{tabular}{|l|cccccccc|}
\hline
Part & $A$ & $A'$ & $B$ & $B'$ & $B''$ & $C$ & $C'$ & $D$ \\
\hline
Degree & $3a$ & $3.2a$ & $a$ & $1.2a$ & $1.4a$ & $1.5a$ & $1.8a$ & $1.6a$ \\
 & $1/3$ & $16/45$ & $1/9$ & $2/15$ & $7/45$ & $1/12$ & $1/5$ & $8/45$ \\
\hline
\end{tabular}
\end{center}
\caption{The degrees of vertices in the nine parts of Rademacher graphon $W_{R}$,
         where $a=1/9$ is the common size of all parts of $W_{R}$ except for $C$.}
\label{tbl-chess}
\end{table}

For $x \in  [0, 1)$, let us denote by $\bit{x}$ the smallest integer $k$ such that $x + 2^{-k} < 1$.
Note that $\bit{x}-1$ is equal to the number of consecutive non-zero bits
after the decimal point in the binary representation of $x$.
We next define the values of $W_{R}(x,y)$ for $(x,y)\in [0,1)^2$, i.e., when both $x$ and $y$ belong to one of the eight parts;
if $x=1$ or $y=1$,
we may set the values arbitrarily (since this concerns the values for a set of measure zero),
for example, to zero.
The value of $W_{R}(x, y)$ is equal to $1$ in the following cases:
\begin{itemize}
\item $x, y \in A$ and $\bit{\frac{x-\ell_A}{a}} \neq \bit{\frac{y-\ell_A}{a}}$,
\item $x, y \in A'$ and $\bit{\frac{x-\ell_{A'}}{a}} \neq \bit{\frac{y-\ell_{A'}}{a}}$,
\item $x \in A$, $y \in A'$ and $\bit{\frac{x-\ell_{A}}{a}} = \bit{\frac{y-\ell_{A'}}{a}}$,
\item $x \in A$, $ y \in B$ and $(x-\ell_{A})+(y-\ell_{B}) \le a$,
\item $x \in A$, $ y \in B''$ and $(x-\ell_{A})+(y-\ell_{B''}) \ge a$,
\item $x \in A'$, $ y \in B'$ and $(x-\ell_{A'})+(y-\ell_{B'})\le a$,
\item $x \in A'$, $ y \in B''$ and $y-\ell_{B''}\le x-\ell_{A'}$,
\item $x, y \in B$ and $(x-\ell_{B}) + (y-\ell_{B}) \ge a$,
\item $x, y \in B'$ and $(x-\ell_{B'}) + (y-\ell_{B'}) \ge a$,
\item $x, y \in C'$ and $(x-\ell_{C'}) + (y-\ell_{C'}) \ge a$,
\item $x \in A$, $ y \in C$ and $\left\lfloor \frac{y-\ell_C}{2a} \cdot 2^{\bit{\frac{x-\ell_{A}}{a}}}\right\rfloor$ is even, and
\item $x \in A'$, $y\in C'$ and  $\left(1-2^{-\bit{\frac{x-\ell_{A'}}{a}}}-\frac{x-\ell_{A'}}{a}\right) \cdot 2^{\bit{\frac{x-\ell_{A'}}{a}}} + \frac{y-\ell_{C'}}{a} \le 1$.
\end{itemize}
If $x\in A'$, $y\in C$ and $\left\lfloor \frac{y-\ell_C}{2a} \cdot 2^{\bit{\frac{x-\ell_{A'}}{a}}}\right\rfloor$ is even,
then
\[W_{R}(x,y)=\left(1-2^{-\bit{\frac{x-\ell_{A'}}{a}}}-\frac{x-\ell_{A'}}{a}\right)\cdot2^{\bit{\frac{x-\ell_{A'}}{a}}}.\]
If $x,y\in C$, then $W_{R}(x,y)=3/4$ if $(x-\ell_C)+(y-\ell_C)\ge 2a$.
If $y\in D$, then
\[W_{R}(x,y)=\left\{
           \begin{array}{cl}
	   0.2 & \mbox{if $x\in A'$ or $x\in B'$,} \\
	   0.4 & \mbox{if $x\in B''$, and } \\
	   0.8 & \mbox{if $x\in C'$.}
	   \end{array}
           \right.\]
Finally, $W_{R}(x,y)$ is zero if neither $(x,y)$ nor the symmetric pair fall in any of the described cases.
It is routine to compute the degrees of vertices in the eight parts of the just defined graphon and
to check that they match the values given in Table~\ref{tbl-chess}.

We finish this section with establishing that
the space of typical vertices of Rademacher graphon $W_R$ is not locally compact as claimed in Theorem~\ref{thm:main}.

\begin{proposition}
\label{prop-non-compact}
The topological space $T(W_{R})$ is not locally compact.
\end{proposition}

\begin{proof}
Let $g:[0,1]\to [0,1]$ be the function defined as follows:
\[g(x)=\left\{
       \begin{array}{cl}
       1 & \mbox{if $x\in A'\cup B''\cup C'$,} \\
       0.2 & \mbox{if $x\in D$, and} \\
       0 & \mbox{otherwise.}
       \end{array}\right.\]
Further, let $g_{i,\delta}:[0,1]\to [0,1]$ for $i\in\NN$ and $\delta\in [0,1]$ be defined as follows:
\[g_{i,\delta}(x)=\left\{
       \begin{array}{cl}
       1 & \mbox{if $x\in A$ and $\bit{\frac{x-\ell_A}{a}}=i$,} \\
       1 & \mbox{if $x\in A'$ and $\bit{\frac{x-\ell_{A'}}{a}}\not=i$,} \\
       1 & \mbox{if $x\in B'$ and $x-\ell_{B'}\le (1+\delta)2^{-i}$,} \\
       1 & \mbox{if $x\in B''$ and $x-\ell_{B''}\le 1-(1+\delta)2^{-i}$,} \\
       \delta & \mbox{if $x\in C$ and $\left\lfloor 2^i\cdot \frac{x-\ell_C}{2a}\right\rfloor$ is even,} \\
       1 & \mbox{if $x\in C'$ and $\frac{x-\ell_{C'}}{a}\le 1-\delta$,} \\
       0.2 & \mbox{if $x\in D$, and} \\
       0 & \mbox{otherwise.}
       \end{array}\right.\]
Observe that $W_{R}\lt(2a-(1+\delta)2^{-i}a,y\rt)=g_{i,\delta}(y)$
for every $i\in\NN$, $\delta\in (0,1)$ and $y\in [0,1]$.
The following two estimates on the $L_1$-distances between $g$ and $g_{i,\delta}$ are straightforward to obtain:
\[
\begin{array}{rcll}
\| g_{i,\delta}-g\|_1 & =&  \frac{(2+2\delta)\cdot 2^{-i}+2\delta}{9}&\mbox{,} \\
\| g_{i,\delta}-g_{i',\delta'}\|_1 & = & \frac{2+2\cdot 2^{-i}}{9}|\delta-\delta'| &\mbox{if $i=i'$, and}\\
\| g_{i,\delta}-g_{i',\delta'}\|_1 & > & \frac{\delta+\delta'}{18} &\mbox{if $i\neq i'$.}
\end{array}
\]
Hence, every neighborhood of $g$ contains $g_{i,\delta}$ with $\delta<\delta_0$ for some $i\in\NN$ and $\delta_0>0$.
Consequently, $g$ belongs to $T(W)$.
An analogous argument yields that $g_{i,\delta}\in T(W)$ for all $i\in\NN$ and $\delta\in (0,1)$.
However, for every $\varepsilon>0$,
all the functions $g_{i,\varepsilon}$ with $i>\log_2\varepsilon^{-1}$
are at $L_1$-distance at most $\varepsilon$ from $g$ and
the $L_1$-distance between any pair of them is at least $\varepsilon/9$.
We conclude that no neighborhood of $g$ in $T(W)$ is compact.
\end{proof}

\section{Constraints}

In this section,
we describe a set $\CC_{R}$ of constraints such that $W_R$ is the unique graphon that satisfies them;
this assertion is then proven in the next section.
For clarity of the exposition,
the constraints are split into eight groups and each is given a name.
\begin{description}
\item[Group 1: Partition constraints.]
  The partition constraints are the (finitely many) constraints given in Lemma~\ref{lm-partition} that
  are satisfied by a graphon $W$ if and only if $W$ is a partitioned graphon with its parts having the same sizes and degrees as those of $W_R$.
\item[Group 2: Zero constraints.]
  The zero constraints are sixteen (non-rooted) decorated constraints of the form $e=0$,
  where $e$ is an edge with its vertices decorated with $X$ and $Y$
  for
  \[\begin{array}{cl}
    (X,Y)\in\{&(B'',B''), (D,D), (A,C'), (A,D), (A',B), (B,B'),\\
              &(B,B''), (B,C), (B,C'), (B,D), (B',B''), (B',C),\\
	      &(B',C'), (B'',C), (B'',C'), (C,D)\;\}.
    \end{array}\]
\item[Group 3: Pseudorandom constraints.]
  The pseudorandom constraints are the rooted decorated constraints given in Lemma~\ref{lm-pseudobipartite} that
  are satisfied by a graphon $W$ that has the same parts as $W_R$ if and only if
  the graphon $W$ is equal to $0.2$ between the parts $D$ and $A'$,
  to $0.2$ between the parts $D$ and $B'$,
  to $0.4$ between the parts $D$ and $B''$, and
  to $0.8$ between the parts $D$ and $C'$.
\item[Group 4: Triangular constraints.]
  For $p\in [0,1]$, let $W_{\blacktriangle,p}$ be the graphon such that $W_{\blacktriangle,p}(x,y)=p$ if $x+y\ge 1$ and
  $W_{\blacktriangle,p}(x,y)=0$ otherwise;
  the graphon $W_{\blacktriangle,p}$ is finitely forcible~\cite[Corollaries 3.15 and 5.2]{bib-lovasz11+} for every $p\in [0,1]$.
  The triangular constraints are the non-rooted decorated constraints given in Lemma~\ref{lm-gadget}
  applied with $W_0=W_{\blacktriangle,1}$ and each of the parts $B$, $B'$ and $C$, and
  with $W_0=W_{\blacktriangle,3/4}$ and the part $C'$.
\item[Group 5: Monotonicity constraints.] The monotonicity constraints are the nine decorated constraints depicted in Figure~\ref{fig-mono}.
\item[Group 6: Split constraints.] The split constraints are the seven decorated constraints depicted in Figure~\ref{fig-split}.
\item[Group 7: Infinitary constraints.] The infinitary constraints are the four decorated constraints depicted in Figure~\ref{fig-infin}.
\item[Group 8: Orthogonality constraints.] The orthogonality constraints are the five decorated constraints depicted in Figure~\ref{fig-ortho}.
\end{description}
By Lemma~\ref{lm-partitioned},
all constraints described above can be expressed as equivalent ordinary constraints (note that
the partition constraints are ordinary constraints by themselves).

\begin{figure}
\begin{center}
\epsfbox{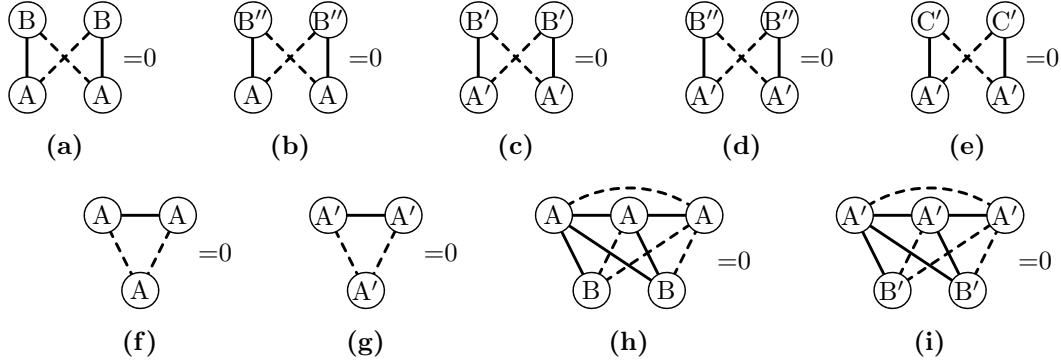}
\end{center}
\caption{The monotonicity constraints.}
\label{fig-mono}
\end{figure}

\begin{figure}
\begin{center}
\epsfbox{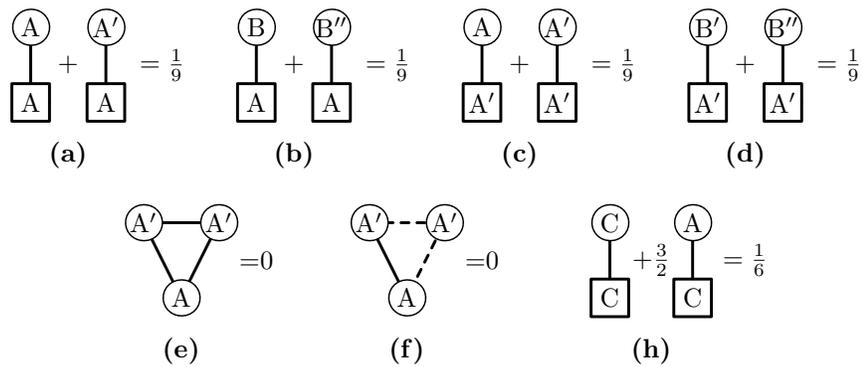}
\end{center}
\caption{The split constraints.}
\label{fig-split}
\end{figure}

\begin{figure}
\begin{center}
\epsfbox{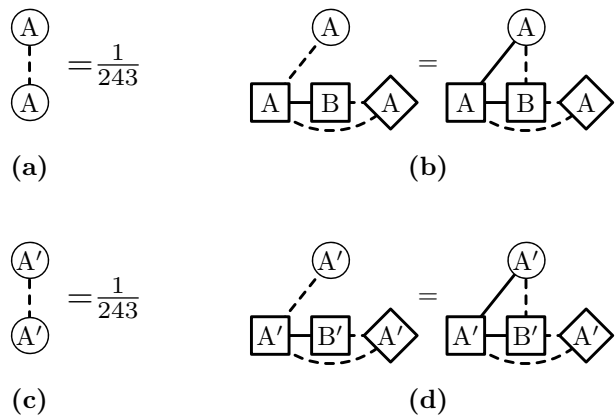}
\end{center}
\caption{The infinitary constraints.}
\label{fig-infin}
\end{figure}

\begin{figure}
\begin{center}
\epsfbox{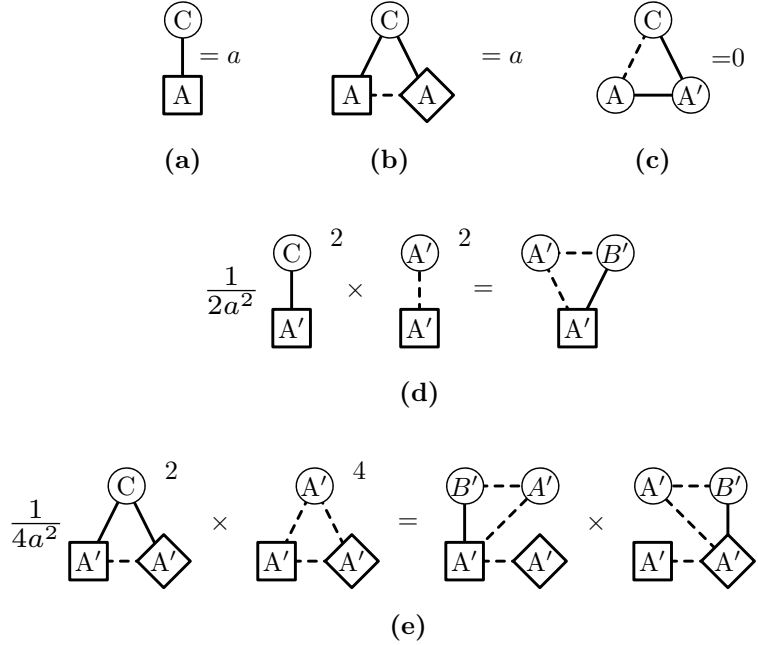}
\end{center}
\caption{The orthogonality constraints. Note that $a=1/9$.}
\label{fig-ortho}
\end{figure}

\section{Forcing}

This section is devoted to proving that Rademacher graphon $W_{R}$ is the unique graphon (up to a weak isomorphism)
satisfying the constraints contained in the set $\CC_R$ presented in the previous section.
The proof is split into several claims, each giving an analysis of the values of a graphon for different parts (see Table~\ref{tbl-plan}).
The analysis will also yield that the graphon $W_{R}$ satisfies all constraints in the set $\CC_R$.

\begin{table}
\begin{center}
\begin{tabular}{c|cccccccc}
& $A$ & $A'$ & $B$ & $B'$ & $B''$ & $C$ & $C'$ & $D$ \\
\hline
%$A$ & \ref{clm-group-7} & \ref{clm-group-7} & \ref{clm-group-5} & \ref{clm-group-2+3} & \ref{clm-group-5} & \ref{clm-group-8a} & \ref{clm-group-2+3} & \ref{clm-group-2+3} \\
%$A'$ & & \ref{clm-group-7} & \ref{clm-group-2+3} & \ref{clm-group-5} & \ref{clm-group-5} & \ref{clm-group-8b} & \ref{clm-group-final} & \ref{clm-group-2+3} \\
%$B$ & & & \ref{clm-group-4} & \ref{clm-group-2+3} & \ref{clm-group-2+3} & \ref{clm-group-2+3} & \ref{clm-group-2+3} & \ref{clm-group-2+3} \\
%$B'$ & & & & \ref{clm-group-4} & \ref{clm-group-2+3} & \ref{clm-group-2+3} & \ref{clm-group-2+3} & \ref{clm-group-2+3} \\
%$B''$ & & & & & \ref{clm-group-2+3} & \ref{clm-group-2+3} & \ref{clm-group-2+3} & \ref{clm-group-2+3} \\
%$C$ & & & & & & \ref{clm-group-4} & \ref{clm-group-2+3} & \ref{clm-group-2+3} \\
%$C'$ & & & & & & & \ref{clm-group-4} & \ref{clm-group-2+3} \\
%$D$ & & & & & & & & \ref{clm-group-2+3}
$A$ & \ref{clm-group-7} & \ref{clm-group-7} & \ref{clm-group-5} & \ref{clm-group-2+3} & \ref{clm-group-5} & \ref{clm-group-8a} & \ref{clm-group-2+3} & \ref{clm-group-2+3} \\
$A'$ & \ref{clm-group-7} & \ref{clm-group-7} & \ref{clm-group-2+3} & \ref{clm-group-5} & \ref{clm-group-5} & \ref{clm-group-8b} & \ref{clm-group-final} & \ref{clm-group-2+3} \\
$B$ & \ref{clm-group-5} & \ref{clm-group-2+3} & \ref{clm-group-4} & \ref{clm-group-2+3} & \ref{clm-group-2+3} & \ref{clm-group-2+3} & \ref{clm-group-2+3} & \ref{clm-group-2+3} \\
$B'$ & \ref{clm-group-2+3} & \ref{clm-group-5} & \ref{clm-group-2+3} & \ref{clm-group-4} & \ref{clm-group-2+3} & \ref{clm-group-2+3} & \ref{clm-group-2+3} & \ref{clm-group-2+3} \\
$B''$ & \ref{clm-group-5} & \ref{clm-group-5} & \ref{clm-group-2+3} & \ref{clm-group-2+3} & \ref{clm-group-2+3} & \ref{clm-group-2+3} & \ref{clm-group-2+3} & \ref{clm-group-2+3} \\
$C$ & \ref{clm-group-8a} & \ref{clm-group-8b} & \ref{clm-group-2+3} & \ref{clm-group-2+3} & \ref{clm-group-2+3} & \ref{clm-group-4} & \ref{clm-group-2+3} & \ref{clm-group-2+3} \\
$C'$ & \ref{clm-group-2+3} & \ref{clm-group-final} & \ref{clm-group-2+3} & \ref{clm-group-2+3} & \ref{clm-group-2+3} & \ref{clm-group-2+3} & \ref{clm-group-4} & \ref{clm-group-2+3} \\
$D$ & \ref{clm-group-2+3} & \ref{clm-group-2+3} & \ref{clm-group-2+3} & \ref{clm-group-2+3} & \ref{clm-group-2+3} & \ref{clm-group-2+3} & \ref{clm-group-2+3} & \ref{clm-group-2+3}
\end{tabular}
\end{center}
\caption{The numbers of the claims where values of the graphon between the given parts are analyzed.}
\label{tbl-plan}
\end{table}

We start with describing the general set up of the proof.
Let $W$ be a graphon satisfying all constraints in $\CC_{R}$;
our aim is to show that $W$ and $W_{R}$ are weakly isomorphic.
To do so, we assume that $W$ satisfies the constraints in Group 1 (the partition constraints) and
construct two measure preserving maps $\varphi,\psi:[0,1]\to [0,1]$,
which will be fixed for the remainder of the section.
The arguments presented further in this section will give that
the graphons $W^{\varphi}$ and $W^{\psi}_{R}$ are equal almost everywhere.

Since the graphon $W$ satisfies the partition constraints,
Lemma~\ref{lm-partition} implies that $W$ is a partitioned graphon with parts of the same sizes and degrees as the parts of $W_R$.
In particular, there exists a measure preserving map $\varphi:[0,1]\to [0,1]$ such that
each of the parts of $W_R$ is mapped by $\varphi$ to the corresponding part of $W$.
To construct the map $\psi$, we use the following statement, which is known as Monotone Reordering Theorem.
\begin{theorem}
\label{thm-MRT}
Let $I$ be a subinterval of $[0,1]$.
For every measurable function $h:I\to\RR$,
there exist a monotone non-decreasing function $f:I\to\RR$ and a measure preserving map $\psi_I:I\to I$ such that
$h(x)=f(\varphi_I(x))$ for almost every $x\in I$.
\end{theorem}
Recall that $A$, $A'$, $B$, $B'$, $B''$, $C$, $C'$ and $D$ are the half-open intervals which form the parts of $W_R$.
The measure preserving map $\psi$ is constructed by applying Theorem~\ref{thm-MRT}
with each of the intervals $A$, ${A'}$, $B$, ${B'}$, ${B''}$, $C$ and ${C'}$
to obtain non-increasing functions $f_A:A\to\RR$ and $f_{A'}:{A'}\to\RR$, and
non-decreasing functions $f_B:B\to\RR$, $f_{B'}:{B'}\to\RR$, $f_{B''}:{B''}\to\RR$,
$f_C:C\to\RR$ and $f_{C'}:{C'}\to\RR$ such that the following holds almost every $x$:
\[\begin{array}{ccclcccl}
\forall x\in A & f_A(\psi(x)) & = & \int\limits_{B} W^{\varphi}(x,y) \dif y &
\forall x\in {A'} & f_{A'}(\psi(x)) & = & \int\limits_{{B'}} W^{\varphi}(x,y) \dif y \\
\forall x\in B & f_B(\psi(x)) & = & \int\limits_{B} W^{\varphi}(x,y) \dif y &
\forall x\in {B'} & f_{B'}(\psi(x)) & = & \int\limits_{{B'}} W^{\varphi}(x,y) \dif y \\
&&&&
\forall x\in {B''} & f_{B''}(\psi(x)) & = & \int\limits_{{A}} W^{\varphi}(x,y) \dif y \\
\forall x\in C & f_C(\psi(x)) & = & \int\limits_{C} W^{\varphi}(x,y) \dif y &
\forall x\in {C'} & f_{C'}(\psi(x)) & = & \int\limits_{{C'}} W^{\varphi}(x,y) \dif y
\end{array}\]
To complete the definition of $\psi$,
we set $\psi(x)=x$ for all $x\in D$ and $\psi(1)=0$.

We are now ready to start the analysis of the graphon $W$,
which will eventually lead to the conclusion that
the graphons $W^{\varphi}$ and $W^{\psi}_{R}$ are equal almost everywhere.

\begin{claim}
\label{clm-group-2+3}
If $W$ satisfies the constraints in Groups 2 and 3 (the zero and pseudorandom constraints),
then $W^{\varphi}$ and $W^{\psi}_{R}$ are equal for almost every pair $(x,y)$
from each of the following sets:
$A\times (B'\cup C')$, $A'\times B$, $B\times B'$, $(B\cup B'\cup B'')\times (B''\cup C\cup C')$, $C\times C'$ and $D\times [0,1)$.
\end{claim}

\begin{proof}
The constraints in Group 2 yield that $W^{\varphi}(x,y)=0$ for almost every $(x,y)$
from the following sets:
$A\times (B'\cup C')$, $B\times (A'\cup B')$, $(B\cup B'\cup B'')\times (B''\cup C\cup C')$, $C\times C'$ and $D\times (A\cup B\cup C\cup D)$.
Lemma~\ref{lm-pseudobipartite} yields that
if $W$ satifies the constraints in Group 3,
then $W^{\varphi}(x,y)=0.2$ for almost every $(x,y)\in D\times (A'\cup B')$,
$W^{\varphi}(x,y)=0.4$ for almost every $(x,y)\in D\times B''$, and
$W^{\varphi}(x,y)=0.8$ for almost every $(x,y)\in D\times C'$.
It follows that the graphons $W^{\varphi}$ and $W^{\psi}_{R}$ are equal almost everywhere
on the sets listed in the statement of the lemma.
\end{proof}

\begin{claim}
\label{clm-group-4}
If $W$ satisfies the constraints in Group 4 (the triangular constraints),
then $W^{\varphi}$ and $W^{\psi}_{R}$ are equal for almost every pair $(x,y)$
from each of the following sets:
$B\times B$, $B'\times B'$, $C\times C$ and $C'\times C'$.
\end{claim}

\begin{proof}
We focus on the analysis of the values of the graphons $W^{\varphi}$ and $W^{\psi}_{R}$ on $B\times B$.
Recall that $B=[2/9,1/3)=[\ell_B,\ell_B+1/9)$.
Lemma~\ref{lm-gadget} and the choice of the triangular constraints decorated with the part $B$
yield that $f_B(\psi(x))=\psi(x)-\ell_B$ for almost every $x\in B$,
$W^{\varphi}(x,y)=1$ for almost every $(x,y)\in B\times B$ with $\psi(x)+\psi(y)\ge 2\ell_B+1/9$, and
$W^{\varphi}(x,y)=0$ for almost every $(x,y)\in B\times B$ with $\psi(x)+\psi(y)<2\ell_B+1/9$.
It follows that the graphons $W^{\varphi}$ and $W^{\psi}_{R}$ are equal for almost every pair $(x,y)\in B\times B$.
The arguments that $W^{\varphi}$ and $W^{\psi}_{R}$ are equal
for almost every pair $(x,y)\in B'\times B'$, $(x,y)\in C\times C$ and $(x,y)\in C'\times C'$ are analogous.
\end{proof}

Before stating the next lemma, we introduce some additional notation.
If $x$ is a vertex and $Y$ is one of the parts,
we write $N_Y(x)$ for the set of $y\in Y$ such that $W^\varphi(x,y)>0$.
Further, if vertices $x$ and $y$ belong to the same part of $W^\varphi$,
then we write $x\preceq y$ iff $\psi(x)\le\psi(y)$.

\begin{claim}
\label{clm-group-5}
If $W$ satisfies the constraints in Groups 2--6,
then $W^{\varphi}$ and $W^{\psi}_{R}$ are equal for almost every pair $(x,y)$
from each of the following two sets: $A\times (B\cup B'')$ and $A'\times (B'\cup B'')$.
\end{claim}

\begin{proof}
The constraint (a) in Figure~\ref{fig-mono} yields that
$W^{\varphi}(x,y)\in\{0,1\}$ for almost every $(x,y)\in A\times B$ and
$|N_B(x')\setminus N_B(x)|=0$ or $|N_B(x)\setminus N_B(x')|=0$ for almost every pair $(x,x')\in A\times A$ (otherwise,
the density of the decorated graph in the constraint would be positive).
Since the function $f_A$ from the definition of $\psi$ is non-increasing,
it follows that $|N_B(x')\setminus N_B(x)|=0$ for almost every $x,x'\in A$ with $x\preceq x'$.
Next,
since the degree of every $y\in B$ is $1/9$ and $W^{\varphi}(x,y)=0$
for almost every $x\not\in A\cup B$ and almost every $y\in B$,
the sets $N_A(y)$ and $\{x\in A, \psi(x)\le \ell_B+1/9-\psi(y)\}$
differ on a set of measure zero for almost every $y\in B$.
We conclude that the graphons $W^{\varphi}$ and $W^{\psi}_{R}$ are equal for almost every pair $(x,y)\in A\times B$.
An analogous argument involving the constraint (c) in Figure~\ref{fig-mono}
yields that $W^{\varphi}$ and $W^{\psi}_{R}$ are equal for almost every pair $(x,y)\in A'\times B'$.

The constraint (b) in Figure~\ref{fig-mono} and
the non-strict monotonicity of the function $f_{B''}$
yield that
$W^{\varphi}(x,y)\in\{0,1\}$ for almost every $(x,y)\in B''\times A$ and
$|N_{A}(x)\setminus N_{A}(x')|=0$ for almost every $x,x'\in B''$ with $x\preceq x'$.
The split constraint (b) in Figure~\ref{fig-split}
implies that $|N_B(y)+N_{B''}(y)|=1/9$ for almost every $y\in A$.
It follows that the sets $N_{B''}(y)$ and $\{x\in B'',\psi(x)\ge \ell_{B'}+1/9-\psi(y)\}$
differ on a set of measure zero for almost every $y\in A$.
In particular, the graphons $W^{\varphi}$ and $W^{\psi}_{R}$ are equal for almost every pair $(x,y)\in B''\times A$.

Since every vertex of $B''$ has degree $7/45$ and
the graphons $W^{\varphi}$ and $W^{\psi}_{R}$ are equal for almost every pair $(x,y)\in B''\times ([0,1]\setminus A')$,
we derive that 
\[\int\limits_{{A'}} W^{\varphi}(x,y) \dif y=\ell_{B'}+1/9-\psi(x)\]
for almost every $x\in B''$.
Using the monotonicity constraint (d) in Figure~\ref{fig-mono} and
the split constaint (d) in Figure~\ref{fig-split},
we derive in a way analogous to that used in the previous paragraph that
the sets $N_{B''}(y)$ and $\{x\in B'',\psi(x)\le\psi(y)-\ell_{B''}\}$
differ on a set of measure zero for almost every $y\in A'$.
Consequently, the graphons $W^{\varphi}$ and $W^{\psi}_{R}$ are equal for almost every pair $(x,y)\in B''\times A'$.
\end{proof}

We now arrive to one of the two most involved lemmas in this section.

\begin{claim}
\label{clm-group-7}
If $W$ satisfies the constraints in Groups 2--7,
then $W^{\varphi}$ and $W^{\psi}_{R}$ are equal for almost every pair $(x,y)\in (A\cup A')\times (A\cup A')$.
\end{claim}

\begin{proof}
By Claim~\ref{clm-group-5},
the graphons $W^{\varphi}$ and $W^{\psi}_{R}$ are equal for almost every pair $(x,y)\in A\times B$.
The monotonicity constraint (f) in Figure~\ref{fig-mono} yields that
$W^{\varphi}(x,y)\in\{0,1\}$ for almost every $(x,y)\in A\times A$ and that
almost every point $x\in A$ can be associated with a set $J_x\subseteq A$ such that
$W^\varphi(x,x')=0$ for almost every $x'\in\psi^{-1}(J_x)$,
$W^\varphi(x,x')=1$ for almost every $x'\in A\setminus\psi^{-1}(J_x)$, and
$W^\varphi(x',x'')=0$ for almost every $x',x''\in\psi^{-1}(J_x)$.
In addition, the monotonicity constraint (h) from Figure~\ref{fig-mono} yields that
the set $J_x$ contains $\psi(x)$ and
differs from an interval on a set of measure zero for almost every $x\in A$ (otherwise,
the density of the decorated graph depicted in the constraint (h) would be positive).
Hence, we can assume that the set $J_x$ is an open interval for every $x\in A$;
note that this interval is uniquely determined  for almost every $x\in A$.
Finally, the constraint (f) yields that
there exists a set $Z$ of measure zero such that
the intervals $J_x$ and $J_{x'}$ are either the same or disjoint for all $x,x''\in A\setminus Z$.
By including additional points to $Z$ while keeping its measure to be zero,
we can assume that the following holds for all points $x\in A\setminus Z$:
$W^\varphi(x,x')=0$ for almost every $x'\in\psi^{-1}(J_x)$,
$W^\varphi(x,x')=1$ for almost every $x'\in A\setminus\psi^{-1}(J_x)$, and
$W^\varphi(x',x'')=0$ for almost every $x',x''\in\psi^{-1}(J_x)$.

Let $\JJ$ be the set of all non-empty intervals $J_x$, $x\in A\setminus Z$.
Since the intervals in $\JJ$ are disjoint, the set $\JJ$ is equipped with a natural linear order.
Note that the above analysis implies that
the following holds for almost every $(x,x')\in A\times A$:
$W^\varphi(x,x')=0$ if and only if there exists an interval $J\in\JJ$ such that
both $\psi(x)$ and $\psi(x')$ are contained in $J$, and
$W^\varphi(x,x')=1$ otherwise.

We now focus on the infinitary constraint (b) from Figure~\ref{fig-infin}.
Fix the leftmost root $x\in A\setminus Z$ such that $|J_x|>0$, and
observe that the set of choices of vertices for the other two roots has non-zero measure unless $\psi(x)=\sup J_x$.
If $\psi(x)<\sup J_x$, consider any such choice of the remaining two roots.
The left hand side of the constraint is equal to the measure of $J_x$, i.e., $\sup J_x-\inf J_x$, and
the right hand side is equal to $1/9-\sup J_x$.
We conclude that $\sup J_x=1/9-|J_x|$ and so $\inf J_x=1/9-2|J_x|$.
This implies that the set $\JJ$ is well-ordered and countable (recall that
the intervals in $\JJ$ are disjoint).

Let $J_k$ be the $k$-th interval contained in $\JJ$.
Furthermore, for $k\geq 1$, define
\[\beta_k= \frac{2(1-9\inf J_{k+1})}{1-9\inf J_k} = \frac{2|J_{k+1}|}{|J_k|}\;\mbox{,}\]
and let $\beta_0$ be equal to $1-9\inf J_1$.
If the set $\JJ$ is finite, define $\beta_k=0$ for $k\ge|\JJ|$.
Since it holds that $\inf J_{k+1}\ge\sup J_k$, we obtain $\beta_k\leq 1$ for every $k\geq 0$.
We can now express the density of non-edges with both end-vertices decorated by $A$ as
\[\sum_{J\in \JJ} |J|^2=\sum_{k=1}^\infty \lt(\frac{1}{9\cdot 2^{k}}\prod\limits_{k'=0}^{k-1}\beta_{k'}\rt)^2\;\mbox{.}\]
The infinitary constraint (a) in Figure~\ref{fig-infin} asserts that
the above sum is equal to $1/243$.
However, this is possible only if $\beta_k=1$ for every $k\ge 0$.
It follows that the set $\JJ$ is infinite and it holds
\[J_k=\lt(\frac{1-2^{-k+1}}{9},\frac{1-2^{-k}}{9}\rt)\]
for every $k\in\NN$.
We conclude that
the graphons $W^\varphi$ and $W^\psi_{R}$ agree almost everywhere on $A\times A$.

A completely analogous argument using
the monotonicity constraints (g) and (i) in Figure~\ref{fig-mono} and
the infinitary constraint (c) and (d) in Figure~\ref{fig-infin} yields that
the graphons $W^\varphi$ and $W^\psi_{R}$ agree almost everywhere on $A'\times A'$.
In particular, there exists an infinite set $\JJ'$ containing intervals
\[J'_k=\lt(\ell_{A'}+\frac{1-2^{-k+1}}{9},\ell_{A'}+\frac{1-2^{-k}}{9}\rt)\]
for all $k\in\NN$, and the following holds for almost every $(x,x')\in A'\times A'$:
$W^\varphi(x,x')=0$ if and only if there exists an interval $J\in\JJ'$ such that
both $\psi(x)$ and $\psi(x')$ are contained in $J$, and
$W^\varphi(x,x')=1$ otherwise.

To complete the proof of the lemma,
we need to establish that the graphons $W^\varphi$ and $W^\psi_{R}$ agree almost everywhere on $A\times A'$.
The split constraints (e) and (f) from Figure~\ref{fig-split}
yield that for almost every $x\in A$ with $\lt|N_{A'}(x)\rt|>0$,
there exists $J'\in \JJ'$ such the following holds:
$W^\varphi(x,y)=1$ for almost every $y\in \psi^{-1}(J')$ and
$W^\varphi(x,y)=0$ for almost every $y\in A'\setminus \psi^{-1}(J')$.
Fix $k\in\NN$.
Since $\lt|N_A(x)\rt|=1-\frac{1}{2^k\cdot 9}$ for almost every $x\in\psi^{-1}(J_k)$,
the split constraint (a) from Figure~\ref{fig-split} implies that
the measure of the interval $J'$ associated with $x$ as given above
is equal $\frac{1}{2^k\cdot 9}$ for almost every $x\in\psi^{-1}(J_k)$.
Hence, for almost every $x\in\psi^{-1}(J_k)$, the interval $J'$ is the interval $J'_k$.
It follows that the following holds for almost every $(x,y)\in A\times A'$:
$W^\varphi(x,y)=1$ if there exists $k$ such that $\psi(x)\in J_k$ and $\psi(y)\in J'_k$, and
$W^\varphi(x,y)=0$ otherwise.
We conclude that the graphons $W^\varphi$ and $W^\psi_{R}$ agree almost everywhere on $A\times A'$.
\end{proof}

We now come to the second main lemma of this section.

\begin{claim}
\label{clm-group-8a}
If $W$ satisfies the constraints in Groups 2--8,
then $W^{\varphi}$ and $W^{\psi}_{R}$ are equal for almost every pair $(x,y)\in A\times C$.
\end{claim}

\begin{figure}
\begin{center}
\epsfbox{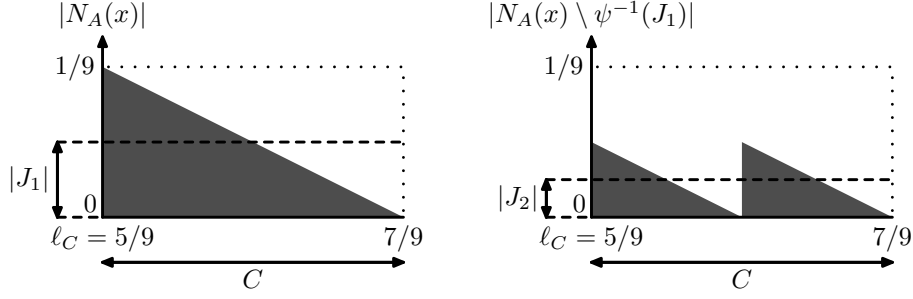}
\end{center}
\caption{Visualization of the argument used in the proof of Claim~\ref{clm-group-8a} to establish that
         the graphons $W^\varphi$ and $W^\psi_{R}$ agree almost everywhere on $A\times C$.}
\label{fig-AvsC}
\end{figure}

\begin{proof}
First note that Claims~\ref{clm-group-4} and~\ref{clm-group-7} guarantee that
the graphons $W^{\varphi}$ and $W^{\psi}_{R}$ agree almost everywhere on $A\times A$ and $C\times C$.
Further, let $\JJ$ be the set of intervals from the proof of Claim~\ref{clm-group-7}.
The orthogonality constraints (a) and (b) from Figure~\ref{fig-ortho} yield that
there exist measurable subsets $I_k\subseteq C$ with $|I_k|=a=1/9$ for every $k\in\NN$ such that
the following holds for almost every $y\in\psi^{-1}(J_k)$:
$N_C(y)$ differs from $I_k$ on a set of measure zero and $W^\varphi(x,y)=1$ for almost every $x\in I_k$.
In particular, the following holds for almost every $(x,y)\in C\times A$:
$W^\varphi(x,y)=1$ if there exists $k\in\NN$ such that $x\in I_k$ and $\psi(y)\in J_k$, and
$W^\varphi(x,y)=0$ otherwise.
Consequently, $\psi^{-1}(J_k)$ is contained in $N_A(x)$ upto a set of measure zero for almost every $x\in I_k$.

Since the function $f_C$ is non-decreasing, Claim~\ref{clm-group-4} implies that
\[f_C(\psi(x))=\frac{3}{4}\lt(\psi(x)-\ell_C\rt)\]
for almost every $x\in C$.
Hence, the split constraint (h) from Figure~\ref{fig-split} yields that
\[|N_A(x)|=\frac{1}{9}-\frac{1}{2}\lt(\psi(x)-\ell_C\rt)\]
for almost every $x\in C$.
Next note that it holds $|N_A(x)|\ge |J_1|=1/18$ for almost every $x\in I_1$
since $N_A(x)$ contains $\psi^{-1}(J_1)$ upto a set of measure zero for almost every $x\in I_1$.
It follows that $I_1$ and the set $\psi^{-1}([\ell_C,\ell_C+1/9])$ differ on a set of measure zero;
the argument is also visualized in Figure~\ref{fig-AvsC}.
We next iterate this argument.
In the next iteration,
we observe that $|N_A(x)\setminus\psi^{-1}(J_1)|\ge |J_2|=1/36$ for almost every $x\in I_2$
since $N_A(x)$ contains $\psi^{-1}(J_2)$ upto a set of measure zero for almost every $x\in I_2$.
Hence, we get that $|J_1|+|J_2|\le |N_A(x)|$ for almost every $x\in I_1\cap I_2$ and
$|J_2|\le |N_A(x)|$ for almost every $x\in I_2\setminus I_1$.
It follows that $I_2$ and the set $\psi^{-1}([\ell_C,\ell_C+1/18]\cup [\ell_C+1/9,\ell_C+3/18])$ differ on a set of measure zero.
For the next iteration,
we observe that $|N_A(x)\setminus\psi^{-1}(J_1\cup J_2)|\ge |J_3|=1/72$ for almost every $x\in I_3$ and
conclude that $I_3$ and the set
$\psi^{-1}([\ell_C,\ell_C+1/36]\cup [\ell_C+1/18,\ell_C+3/36]\cup [\ell_C+1/9,\ell_C+5/36]\cup [\ell_C+3/18,\ell_C+7/36])$ differ on a set of measure zero.
In general, we obtain that the set $I_k$ differs from the preimage with respect to $\psi$ of the set
\[\bigcup_{i=1}^{2^{k-1}}\left[\ell_C+\frac{2i-2}{9\cdot 2^{k-1}},\ell_C+\frac{2i-1}{9\cdot 2^{k-1}}\right]\]
on a set of measure zero for every $k\in\NN$.
Hence,
the graphons $W^\varphi$ and $W^\psi_{R}$ agree almost everywhere on $A\times C$.
\end{proof}

We next analyze the values of the graphon, which causes the space of the typical vertices not to be locally compact.

\begin{claim}
\label{clm-group-8b}
If $W$ satisfies the constraints in Groups 2--8,
then $W^{\varphi}$ and $W^{\psi}_{R}$ are equal for almost every pair $(x,y)\in A'\times C$.
\end{claim}

\begin{proof}
First note that Claims~\ref{clm-group-7} and~\ref{clm-group-8a} guarantee that
the graphons $W^{\varphi}$ and $W^{\psi}_{R}$ agree almost everywhere on $(A\cup A')\times (A\cup A')$ and $A\times C$.
Let $\JJ$ and $\JJ'$ be the sets of intervals from the proof of Claim~\ref{clm-group-7} and
let $I_k$, $k\in\NN$, be the sets from the proof of Claim~\ref{clm-group-8a}.
To prove the assertion of this lemma, we need to show that the following holds for almost every $(x,y)\in A'\times C$:
\begin{equation}
W^\varphi(x,y)=\frac{1-2^{-\bit{\frac{\psi(x)-\ell_{A'}}{a}}}-\frac{\psi(x)-\ell_{A'}}{a}}{2^{-\bit{\frac{\psi(x)-\ell_{A'}}{a}}}}
\label{eq-goal-8b}
\end{equation}
if there exists $k\in\NN$ such that $\psi(x)\in J'_k$ and $y\in I_k$, and
$W^\varphi(x,y)=0$ otherwise.

The orthogonality constraint (c) from Figure~\ref{fig-ortho} implies that
the following holds for every $k\in\NN$:
the set $N_C(x')$ is a subset of $N_C(x)$ up to a set of measure zero for almost every
$x\in\psi^{-1}(J_k)$ and $x'\in\psi^{-1}(J'_k)$.
Note that $N_C(x)$ differs from $I_k$ on a set of measure zero for almost every $x\in\psi^{-1}(J_k)$.
Hence,
$W^\varphi(x,y)=0$ for almost every $x\in \psi^{-1}(J'_k)$ and $y\in C\setminus I_k$.

We next interpret the orthogonality constraint (d) from Figure~\ref{fig-ortho}.
Fix an integer $k\in\NN$ and a typical vertex $x\in \psi^{-1}(J'_k)$.
The first term in the product on the left hand side of the constraint
is equal to 
\[\left(\int\limits_C W^{\varphi}(x,y)\dif y\right)^2 =\left( \int\limits_{I_k} W^{\varphi}(x,y)\dif y\right)^2\;\mbox{.}\]
The second term in the product is equal to
\[|J'_k|^2=\left(2^{-\bit{\frac{\psi(x)-\ell_{A'}}{a}}}\cdot a\right)^2.\]
The term on the right hand side is equal to the probability that
random $x'$ and $y$ chosen uniformly and independently from $[0,1]$ satisfy $x'\in \psi^{-1}(J'_k)$, $y\in B'$, and $\psi(x)<\psi(y)<\psi(x')$.
This probability is equal to
\[\frac{a^2}{2}\left(1-2^{-\bit{\frac{\psi(x)-\ell_{A'}}{a}}}-\frac{\psi(x)-\ell_{A'}}{a}\right)^2\;\mbox{.}\]
We deduce that almost every $x\in\psi^{-1}(J'_k)$ satisfies
\begin{equation}
\int\limits_{I_k} W^{\varphi}(x,y)\dif y=\frac{1-2^{-\bit{\frac{\psi(x)-\ell_{A'}}{a}}}-\frac{\psi(x)-\ell_{A'}}{a}}{2^{-\bit{\frac{\psi(x)-\ell_{A'}}{a}}}}\cdot a\;\mbox{.}
\label{eq-cs1}
\end{equation}
An analogous reasoning for the orthogonality constraint (e) given in Figure~\ref{fig-ortho} yields that
almost every pair of vertices $x,x'\in\psi^{-1}(J'_k)$ satisfies
\begin{eqnarray*}
&\frac{1}{4a^2}\cdot\left(\int\limits_{I_k}W^{\varphi}(x,y)W^{\varphi}(x',y)\dif y\right)^2\cdot
\left(2^{-\bit{\frac{\psi(x)-\ell_{A'}}{a}}}\right)^4 =\\
&\qquad\frac{1}{4}\left(1-2^{-\bit{\frac{\psi(x)-\ell_{A'}}{a}}}-\frac{\psi(x)-\ell_{A'}}{a}\right)^2\cdot
\left(1-2^{-\bit{\frac{\psi(x')-\ell_{A'}}{a}}}-\frac{\psi(x')-\ell_{A'}}{a}\right)^2\;\mbox{.}
\end{eqnarray*}
In the same way as in the proof of Lemma~\ref{lm-pseudobipartite},
we obtain from the above equality that the following holds for almost every $x\in\psi^{-1}(J'_k)$:
\begin{equation}
\left(\int\limits_{I_k} W^{\varphi}(x,y)^2\dif y\right)^{1/2}=\frac{1-2^{-\bit{\frac{\psi(x)-\ell_{A'}}{a}}}-\frac{\psi(x)-\ell_{A'}}{a}}{3\cdot 2^{-\bit{\frac{\psi(x)-\ell_{A'}}{a}}}}\;\mbox{.}
\label{eq-cs2}
\end{equation}
Using Cauchy-Schwarz Inequality,
we deduce from (\ref{eq-cs1}) and (\ref{eq-cs2}) (recall that $|I_k|=a$) that
(\ref{eq-goal-8b}) holds for almost every $x\in\psi^{-1}(J'_k)$ and $y\in I_k$.
Hence, the graphons $W^{\varphi}$ and $W^\psi_{R}$ agree almost everywhere on $A'\times C$.
\end{proof}

It remains to show that the graphons $W^{\varphi}$ and $W^\psi_{R}$ agree almost everywhere on $A'\times C'$.

\begin{claim}
\label{clm-group-final}
If $W$ satisfies the constraints in Groups 2--8,
then $W^{\varphi}$ and $W^{\psi}_{R}$ are equal for almost every pair $(x,y)\in A'\times C'$.
\end{claim}

\begin{proof}
The monotonicity constraint (e) from Figure~\ref{fig-mono} yields that
$W^{\varphi}(x,y)\in\{0,1\}$ for almost every $(x,y)\in A\times C'$ and
at least one of the sets $N_{C'}(x)\setminus N_{C'}(x')$ or $N_{C'}(x')\setminus N_{C'}(x)$ has measure zero
and almost every pair $x,x'\in A'$.
Claims~\ref{clm-group-2+3}, \ref{clm-group-5}, \ref{clm-group-7} and~\ref{clm-group-8b}
imply that the graphons $W^{\varphi}$ and $W^{\psi}_{R}$ are equal for almost every pair $(x,y)\in A'\times ([0,1]\setminus C')$.
Since every vertex $x\in A'$ has degree $16/45$, we obtain that
\[|N_{C'}(x)|=\frac{1}{9}\left(1-2^{-\bit{\frac{\psi(x)-\ell_{A'}}{a}}}-\frac{\psi(x)-\ell_{A'}}{a}\right)\cdot 2^{\bit{\frac{\psi(x)-\ell_{A'}}{a}}}\]
for almost every $x\in A'$.
Similarly,
Claims~\ref{clm-group-2+3} and~\ref{clm-group-4} give that they are equal for almost every pair $(x,y) C'\times ([0,1]\setminus A)$,
which yields that $|N_{A'}(y)|=1/9-(\psi(y)-\ell_{C'})$ for almost every $y\in C'$.
It follows that $N_C'(x)$ and
differs from the preimage with respect to $\psi$ of the set
\[\left[\ell_{C'},\ell_{C'}+\frac{1}{9}\left(1-2^{-\bit{\frac{\psi(x)-\ell_{A'}}{a}}}-\frac{\psi(x)-\ell_{A'}}{a}\right)\cdot 2^{\bit{\frac{\psi(x)-\ell_{A'}}{a}}}\right)\]
on a set of measure zero for almost every $x\in A'$.
Hence, the graphons $W^{\varphi}$ and $W^\psi_{R}$ agree almost everywhere on $A'\times C'$.
\end{proof}

Claims~\ref{clm-group-2+3}--\ref{clm-group-final} yield the following.

\begin{corollary}
\label{cor-main}
The graphon $W_{R}$ is finitely forcible.
\end{corollary}

Proposition~\ref{prop-non-compact} and Corollary~\ref{cor-main} together give a proof of Theorem~\ref{thm:main}.

\section{Conclusion}
\label{sec:concl}

It is quite clear that the construction of Rademacher graphon can be modified
to yield other graphons $W$ with non-compact $T(W)$. Some of these modifications
can yield such graphons with a smaller number of parts at the expense of
making the argument that the graphon is finitely forcible less transparent.

In \cite{bib-lovasz11+}, finite forcibility is considered inside two classes of functions.
Conjecture~\ref{conj:1}, which we addressed in this paper, relates to the class they refer to as $\WW_0$.
This class consists of symmetric measurable functions from $[0,1]^2$ to $[0,1]$.
A larger class referred to as $\WW$ in~\cite{bib-lovasz11+}
is the class containing all symmetric measurable functions from $[0,1]^2$ to $\RR$.
We remark that our arguments can be extended to show that
Rademacher graphon $W_R$ is also
finitely forcible inside this larger class.
Also note that stronger constraints involving multigraphs were used in~\cite{bib-lovasz11+} 
but we have used only constraints involving simple graphs in this paper.

In \cite{bib-lovasz-book}, an analogue of the space $T(W)$ with respect to the following metric is considered.
If $f,g\in L_1[0,1]$, then
\[d_W(f,g):=\int\limits_{[0,1]}\left|\;\int\limits_{[0,1]} W(x,y)(f(y)-g(y))\dif y\right| \dif x\;\mbox{.}\]
It is interesting to note that $T(W)$ is always a compact space when metrized by $d_W$~\cite[Corollary 13.28]{bib-lovasz-book}.

Since every non-compact subset of $L_1[0,1]$ has infinite Minkowski dimension (with the same metric as $L_1[0,1]$),
Theorem~\ref{thm:main} also provides a partial answer to~\cite[Conjecture 10]{bib-lovasz11+},
stated here as Conjecture~\ref{conj:2}.
However, the dimension is finite when several other notions of dimension are considered, and
so we do not claim to disprove this conjecture in this paper.
In~\cite{bib-our-next-paper}, the first two authors and Klimo\v sov\'a disprove Conjecture~\ref{conj:2}
in a more convincing way: they construct a finitely forcible graphon $W$
such that a subspace of $T(W)$ is homeomorphic to $[0,1]^\infty$.
The graphon constructed in~\cite{bib-our-next-paper} also has infinite Minkowski dimension
with respect to the metric $d_W$, which has implications on the sizes of its weak regularity partitions~\cite{bib-lovasz-book,bib-lovasz07+}.
A construction of finitely forcible graphons that require weak regularity partitions
with the number of parts almost matching the existing tight lower bound was then given in~\cite{bib-third-paper}.
Both constructions are based on partitioned graphons used in this paper.
As we have mentioned in Section~\ref{sec:intro},
this line of research led to general framerworks for constructing complex finitely graphons presented in~\cite{bib-CKM} and~\cite{bib-KLNS}.

We finish by presenting a construction of a finitely forcible graphon $W_d$
with a part of $T(W_d)$ positive measure isomorphic to $[0,1]^d$;
the construction is analogous to one found earlier by Norine~\cite{bib-norine-comm}.
Fix a positive integer $d$. We construct
a graphon $W_d$ with $2d+2$ parts $A$, $B_1,\ldots,B_{2d}$, and $C$, each of
size $(2d+2)^{-1}$. If $x,y\in B_i$, then $W_d(x,y)=1$ if $x+y\ge (2d+2)^{-1}$,
i.e., $W_d$ is the half graphon on each $B_i^2$.
If $x\in B_i$ and $y\in C$, then $W_d(x,y)=W_d(y,x)=i/4d$.
Fix now a measure preserving map $\varphi$ from $[0,1]$ to $[0,1]^d$.
If $x\in A$ and $y\in B_i$, $i\le d$, then $W_d(x,y)=W_d(y,x)=1$ if $\varphi((2d+2)x)_i\ge (2d+2)y$.
Finally,
if $x\in A$ and $y\in B_i$, $i\ge d+1$, then $W_d(x,y)=W_d(y,x)=1$ if $1-\varphi((2d+2)x)_i\ge (2d+2)y$.
The graphon $W_d$ is equal to zero for other pairs of vertices.
Clearly, $W_d$ is a partitioned graphon with $2d+2$ parts with vertices inside each part having the same degree and
vertices in different parts having different degrees.
Using the techniques presented in this paper and generalizing arguments from~\cite{bib-kral12+},
one can show that $W_d$ is finitely forcible.
Since the subspace of $T(W_d)$ formed by typical vertices from $A$ is homeomorphic
to $[0,1]^d$, the Lebesgue dimension of $T(W_d)$ is at least $d$ (and
the same is true for most notions of a dimension of a topological space).
This shows that finitely forcible graphons can have arbitrarily large finite dimension.

\section*{Acknowledgements}

The authors would like to thank Jan Hladk\'y, Tereza Klimo\v sov\'a, Serguei
Norine, Vojt\v ech T\r{u}ma for their valuable comments on the topics discussed in the paper, and
to the anonymous referee for the detailed comments on the original version of the paper,
which have resulted in a significant improvement of the presentation of the results in the paper.

\end{document}